\documentclass[preprint,12pt]{elsarticle}
\usepackage[utf8]{inputenc}

\usepackage{fullpage}

\usepackage{xcolor}
\usepackage{amsmath}
\usepackage{amssymb}
\usepackage{latexsym}
\usepackage{amsthm}
\usepackage{mathrsfs}
\usepackage{xfrac}

\usepackage{tikz}
\usetikzlibrary{arrows,shapes,matrix,decorations.pathmorphing,shapes.geometric}

\tikzstyle{arc}=[->,shorten <=3pt, shorten >=3pt,
                 >=stealth, line width=1.1pt]
\tikzstyle{edge}=[shorten <=2pt, shorten >=2pt,
                  >=stealth, line width=1.1pt]
\tikzstyle{vertex}=[circle, fill=white, draw,
                    minimum size=7pt,
                    inner sep=0pt, outer sep=0pt]

\usepackage{url}

\newtheorem{theorem}{Theorem}
\newtheorem{lemma}[theorem]{Lemma}
\newtheorem{corollary}[theorem]{Corollary}
\newtheorem{proposition}[theorem]{Proposition}
\newtheorem{observation}[theorem]{Observation}

\newtheorem{question}[theorem]{Question}
\newtheorem{problem}[theorem]{Problem}

\newcommand{\la}{\leftarrow}
\newcommand{\ra}{\rightarrow}

\usepackage{minitoc}

\usepackage{ifpdf}

\journal{Elsevier}

\begin{document}

\begin{frontmatter}



\title{{Oriented expressions of graph properties.}\tnoteref{t1}}

\tnotetext[t1]{This research was supported by SEP-CONACYT grant
A1-S-8397 and DGAPA-PAPIIT grant IA104521.}

\author[FC]{Santiago Guzm\'an-Pro}
\ead{sanguzpro@ciencias.unam.mx}

\author[FC]{C\'esar Hern\'andez-Cruz\corref{cor1}}
\ead{chc@ciencias.unam.mx}

\address[FC]{Facultad de Ciencias\\
Universidad Nacional Aut\'onoma de M\'exico\\
Av. Universidad 3000, Circuito Exterior S/N\\
C.P. 04510, Ciudad Universitaria, CDMX, M\'exico}

\cortext[cor1]{Corresponding author}

\begin{abstract}

Several graph properties are characterized as the class of graphs that
admit an orientation avoiding finitely many oriented structures.
For instance, if $F_k$ is the set of homomorphic images of the directed
path on $k+1$ vertices, then a graph is $k$-colourable if and only if it admits
an orientation with no induced oriented graph in $F_k$.
There is a fundamental question underlying this kind of characterizations:
given a graph property, $\mathcal{P}$, is there a finite set of oriented graphs,
$F$, such that a graph belongs to $\mathcal{P}$ if and only if it admits an
orientation with no induced oriented graph in $F$?
We address this question by exhibiting necessary conditions upon certain
graph classes to admit such a characterization. Consequently, we exhibit an
uncountable family of hereditary classes, for which no such
finite set exists. In particular, the class of graphs with no holes of prime
length belongs to this family.

\end{abstract}


\end{frontmatter}



\section{Introduction}
\label{sec:Introduction}

All graphs and digraphs considered in this work are
loopless and without parallel edges or parallel arcs.
For basic terminology and notation, we refer the
reader to \cite{bondy2008}. Moreover, for any concepts
related to graph and digraph homomorphisms, we refer
the reader to \cite{hell2004}

Given a pair of (oriented) graphs, $G$ and $H$, we will write
$H < G$ if $H$ is an induced (oriented) subgraph of $G$.  A
natural way to characterize hereditary properties is by finding
a minimal set of forbidden induced subgraphs. Such a set is often
called the \textit{set of minimal obstructions} of the associated
hereditary property. Most of the time these sets turn out to be
infinite, hard to describe, or very difficult to find. Nonetheless,
characterizing such properties through alternative forbidden
structures has usually led to a finite set of forbidden structures.
For instance, consider the class of chordal graphs, and denote by
$B_1$ the oriented graph $(\{1,2,3\},\{(1,2),(1,3)\})$. Clearly, the family
of chordal graphs is a hereditary property with an infinite set of minimal
obstructions, but chordal graphs are characterized as those graphs that
admit a $B_1$-free acyclic orientation \cite{skrienJGT6}, i.e.\ a graph
$G$ is chordal if and only if it admits an acyclic orientation $G'$ such
that $B_1$ is not an induced subdigraph of $G'$.
Similarly, the Roy-Gallai-Hasse-Vitaver Theorem \cite{gallaiPCT,
hasseIMN28, royIRO1, vitaverDAN147} states that a graph is $k$-colourable
if and only if it admits an orientation with no directed walk on $k+1$
vertices. This work studies these kind of characterizations, that is,
characterizations of hereditary properties by forbidding certain
orientations.

For a set of oriented graphs, $F$, Skrien defined the class of $F$-graphs
as those graphs that admit an $F$-free orientation \cite{skrienJGT6}.
We believe this definition might be misleading in the sense that the class of
$F$-graphs is negatively defined with respect to $F$. We propose to invert
this definition. Given a class of oriented graphs $\mathcal{O}$ we define
the class of \textit{$\mathcal{O}$-graphs} as the family of underlying graphs
of $\mathcal{O}$. In other words, a graph $G$ is an $\mathcal{O}$-graph
if and only if there is an orientation $G'$ of $G$ such that $G'\in \mathcal{O}$.
For instance, in a basic graph theory course \cite{bondy2008} the student
learns that a graph $G$ is $2$-edge-connected, if and only if it admits a
strongly connected orientation. So if $\mathcal{O}$ is the class of strongly
connected oriented graphs,  then the class of $\mathcal{O}$-graphs is the
family of $2$-edge-connected graphs.  We are interested in classes of
$\mathcal{O}$-graphs where $\mathcal{O}$ is a hereditary property of
(acyclic) oriented graphs with finitely many forbidden substructures.

Consider a pair of (oriented) graphs, $G$ and $H$, if $G$ is
homomorphic to $H$, we will write $G \to H$, and we say that
$G$ is $H$-colourable. A \textit{homomorphism class} is a
class of graphs defined as those graphs homomorphic
to some fixed  graph $H$. Given a set of
oriented graphs, $F$, we denote by $Forb(F)$ the class of
oriented graphs, $G$, such that $H\not\to G$ for every
$H\in F$. For this work, an embedding is a homomorphism
$\varphi\colon G\to H$ such that $G$ is isomorphic to its image, $\varphi[G]$.
So $G$ embeds in $H$ if and only if $G < H$. We extend the
previously introduced notation and denote by $Forb_e(F)$ the class of oriented
graphs, $G$, such that $H\not< G$ for every $H\in F$. Note that the class of
$Forb_e(F)$-graphs corresponds to the class of $F$-graphs in the sense of
Skrien \cite{skrienJGT6}. We denote
by $Forb_e^\ast(F)$  the subclass of acyclic oriented graphs in $Forb_e(F)$.
In any of these cases, if $F$ is the singleton $\{L\}$, we will simply write
$Forb(L)$, $Forb_e(L)$ or $Forb_e^\ast(L)$. Going back to our previous
examples, the class of $Forb(\overrightarrow{P}_{k+1})$-graphs is the
class of $k$-colourable graphs, and a graph $G$ is a $Forb_e^\ast(B_1)$%
-graph if and only if $G$ is a chordal graph.

We say that a hereditary property, $\mathcal{P}$, is \textit{expressible by
$Forb$-graphs}, if there is a finite set of oriented graphs, $F$, such that
$\mathcal{P}$ corresponds to the class of $Forb(F)$-graphs.  In this
case we say that $Forb(F)$ is an \textit{expression of $\mathcal{P}$}.
Now suppose that $Forb(F)$ is an expression of $\mathcal{P}$,
and let $F'$ be the set of all homomorphic images of oriented graphs in
$F$. Then, $\mathcal{P}$ equates the class of $Forb_e(F')$-graphs. So
the expressive power of $Forb_e(F)$-graph classes is more robust than
the expressive power of $Forb(F)$-graph classes. We say that $\mathcal{P}$
is \textit{expressible by forbidden orientations} if there is a finite set,
$F$, of oriented graphs such that $\mathcal{P}$ coincides with
the class of $Forb_e(F)$-graphs. With a simple cardinality argument
one can notice that not every hereditary property is expressible by
forbidden orientations; there are countably many finite sets of oriented
graphs, while there is an uncountable class of hereditary properties.
This simple observations raises the fundamental question from which this
paper stems.

\begin{question}\label{q:basic}
Which hereditary properties are expressible by forbidden orientations?
\end{question}

It is very likely that Skrien did not have this question in mind when
working on \cite{skrienJGT6}. Nonetheless, his work provides
the first partial answer to this question: he provides a list of
graph classes that coincide to a class of $Forb_e(F)$-graphs when
$F$ is a set of orientations of $P_3$. Recently,
we extended his work by considering all sets of oriented graphs
on three vertices \cite{guzmanAR1}. The aforementioned manuscripts
illustrate one way to tackle Question~\ref{q:basic}: fix a finite set
of oriented graphs, $F$, and characterize the class of $Forb_e(F)$-graphs.
A second way to tackle this question is by fixing a hereditary property,
$\mathcal{P}$, and then (be lucky enough to) find a finite
set of oriented graphs, $F$, such that $\mathcal{P}$ corresponds to
the class of $Forb_e(F)$-graphs. We believe that the Roy-Gallai-Hasse%
-Vitaver Theorem can be considered the first result that aligns with
this approach. We certainly did so in \cite{guzmanAR2} where we
showed that for every odd cycle, $C$, there is an oriented path,
$P_C$, such that a graph $G$ is homomorphic to $C$ if and only
if $G$ is a $Forb(P_C)$-graph. Following the perspective behind the previously
mentioned results, in Section~\ref{sec:Forb} we propose a characterization
of those graphs, $H$, for which the class of $H$-colourable graphs is
expressible by $Forb$-orientations.  A third proceeding towards solving
Question~\ref{q:basic} is exhibiting hereditary properties that are not expressible
by forbidden orientations. As far as we are concerned, this is the
first write up to follow this path. Moreover, we believe that the main
contribution of this work is proposing strong necessary condition
upon certain graph classes to be expressible by forbidden orientations.
As a consequence of this result, we obtain an uncountable family
of hereditary properties that are not expressible by forbidden orientations.

As we will see, it turns out that some natural properties are not expressible
by  forbidden orientations, but they do correspond to some class of
$Forb_e^\ast(F)$-graphs for some finite set $F$. If such a finite set,
$F$, exists for some property, $\mathcal{P}$, we will say that
$\mathcal{P}$ is \textit{expressible by forbidden acyclic orientations}.
Some readers might consider this to be a devious tactic to obtain a finite
expression of the corresponding graphs class, as we are forbidding an
infinite set of  oriented graphs beforehand (all directed cycles). Rather
than dwelling on the validity of such technique, we will notice that most
of the results regarding forbidden orientations can be adapted to forbidden
acyclic orientations.

The structure of this work is as follows.  In Section~\ref{sec:Languages}, we
introduce some concepts and prove a couple of results of language theory
that we will use in Section~\ref{sec:Fgraphs}. Consequently, in
Section~\ref{sec:Fgraphs} we present our main results; we propose some
necessary conditions for a hereditary property to be expressible by forbidden
(acyclic) orientations, and we exhibit an uncountable family of hereditary
properties that are not expressible by forbidden (acyclic) orientations.
In Section~\ref{sec:Forb} we propose a characterization
of those homomorphism classes expressible by $Forb$-graphs.
Finally, in Section~\ref{sec:Conc} we present conclusions
and some problems that we think would be an interesting
follow up in this subject.

\section{Languages}\label{sec:Languages}

Consider a finite set $\mathcal{A}$, which we will call
an \textit{alphabet}. A \textit{word over} $\mathcal{A}$
is a finite sequence of elements in $\mathcal{A}$.
Whenever there is no ambiguity on the alphabet, we will
call a word over $\mathcal{A}$ only a \textit{word}.
The \textit{length} of a word $a$ is the number of elements in the sequence,
and we denote it by $|a|$. If $|a| = k$, we say that $a$ is a \textit{$k$-word}.
We denote by $\mathcal{A}_k$ the set of $k$-words over $\mathcal{A}$,
and by $\mathcal{A}^\ast$ the union of $\{\mathcal{A}_k\}_{k\ge 0}$, where
the only word in $\mathcal{A}_0$ is denoted by $\epsilon$ and it is called
the \textit{empty word}. A \textit{language} over $\mathcal{A}$ is a subset
$\mathcal{L}$ of $\mathcal{A}^\ast$.  Consider a pair words, $a$ and $b$,
$a=a_1a_2\dots a_k$ and $b = b_1b_2\dots b_l$, we denote by $ab$ the
word $a_1a_2\dots a_kb_1b_2\dots b_l$. For $n\ge 1$, we denote by $a^n$
the \textit{$nth$-power} of a word $a$, and it is defined recursively; $a^0 = \epsilon$ and
$a^n = a^{n-1}a$. We say that a word $a\in\mathcal{L}$ is \textit{$k$-periodic in
$\mathcal{L}$}, if $a^n\in\mathcal{L}$ for every $n\ge 1$, and $|a| = k$.
We say that $a$ is a \textit{constant word} if all the symbols in $a$ are the same.

We say that a word $a$ is a \textit{factor} of a word $b$
if there are two (possibly empty) words $c$ and $d$ such
that $cad = b$; in this case we write $a<b$. It is
not hard to notice that the relation induced by factors is a partial order in
$\mathcal{A}^\ast$. Let $A$ be a set of words, we say that $b$ is
\textit{$A$-free} if $b$ contains no factor in $A$. We denote the language
of $A$-free words by $\mathcal{L}_A$.  A language is \textit{hereditary} if it
is closed under factors. Note that for any set of words $A$ the language of
$A$-free words is hereditary. Let $m$ be a positive integer,  we say that a
language $\mathcal{L}$ is \textit{$m$-synchronizing}, if for any choice of words $a, b, d$ such that
$ba,ad\in\mathcal{L}$, if $|a| = m$, then $bad\in\mathcal{L}$.
Clearly, $\mathcal{L}$ is $m$-synchronizing if and only if for any pair
of words $ba,ad\in\mathcal{L}$, such that $|a|\ge m$ then $bad\in\mathcal{L}$.

\begin{observation}\label{finitesync}
Let $A$ be a finite set of words over any alphabet, and let $m$ a positive
integer, such that for every $a\in A$, $|a|\le m$, then $\mathcal{L}_A$ is
$m$-synchronizing.
\end{observation}

We can naturally translate words over the alphabet $\{\la,\ra\}$ to oriented
paths. Denote by $W^k$ the set of oriented paths on $k$ edges, and by
$W^\ast$ the union of $\{W_k\}_{k \ge 0}$. Consider the surjective function
$t:\{\la,\ra\}^\ast\to W^\ast$, $t(a_1a_2\dots a_k) = P$,
where $P = v_1v_2\dots v_{k+1}$ such that $v_i\to v_{i+1}$
if $a_i = \ra$, and $v_{i+1}\to v_i$ otherwise. Clearly
$t$ is a monotone function, i.e., for
any $\{\la,\ra\}$-words, $a$ and $b$, if $a<b$ then $t(a)<t(b)$. Moreover, for
any oriented paths, $P$ and $Q$, if $P<Q$, then there are $\{\la,\ra\}$-words,
$p$ and $q$, such that $p<q$, $t(p) = P$, and $t(q) = Q$. The first item
of the following lemma follows from the three previous observations.

\begin{lemma}\label{lem:translatingProperties}
Let $F$ be a finite set of connected oriented graphs. Then, there is a set $A$ of
words over $\{\la,\ra\}$ and a positive integer $m$, such that for any positive
integer $k$ the following statements hold:
\begin{enumerate}
	\item if $k\ge m$, the path on $k$ edges admits an $F$-free
		orientation, if and only if there is a $k$-word in $\mathcal{L}_A$,
	\item if $k\ge \max\{m,4\}$, the cycle on $k$ edges admits an $F$-free
		orientation, if and only if there is a $k$-periodic word in $\mathcal{L}_A$,
		and
	\item if $k\ge \max\{m,4\}$, the cycle on $k$ edges admits an $F$-free
		acyclic orientation, if and only if there is a non-constant $k$-periodic
		word in $\mathcal{L}_A$.
\end{enumerate}
\end{lemma}
\begin{proof}
Let $m$ be the integer $\max\{|V(G)|\colon G\in F\}+1$, and let $A$ be the
set $\{a\in\{\la,\ra\}^\ast\colon t(a)\in F\}$. The first statement follows directly
using the observations preceding this lemma. Now we prove the third
statement. Let $k$ be an integer, $k\ge \max\{m,4\}$, and let
$C = (c_1,c_2,\dots,c_k,c_1)$ be non-directed oriented cycle.
First note that, since $k\ge \max\{m,4\}$, and by the choice of $m$, if there is
a graph $P\in F$ such that $P<C$, then $P$ is an oriented path. Consider the
$k$-word $l = l_1\dots l_k$, where for every $i\in\{1,\dots, k-1\}$, $l_i = \ra$ if
$c_i\to c_{i+1}$, and $l_i = \la$ otherwise, and $l_k = \ra$ if $c_k\to c_0$ and
$l_k=\la$ otherwise. Since $C$ is not a directed cycle, then $l$ is not a constant
word. Moreover, it is not hard to notice that there is a path $P\in F$ such that $P< C$
if and only if there is word $a\in A$ such that $a<l^2$. Similarly, if $h$ is a non-%
constant $k$-word such that $h^2\in \mathcal{L}_A$ we can find an $F$-free
acyclic orientation of the cycle on $k$-edges. Thus, the $k$-cycle $C_k$ admits
an $F$-free acyclic orientation, if and only if there is a non-constant $k$-word,
$l$, such that $l^2\in \mathcal{L}_A$. Now, note that if there is a non-constant
$k$-periodic word in $\mathcal{L}_A$, in particular there is a non-constant
$k$-word $l$ such that $l^2\in \mathcal{L}_A$. The converse implication also
holds since every word in $A$ is bounded by $m$, and $k\ge m$, so by
Observation~\ref{finitesync}, if $l^2\in\mathcal{L}_A$ then $l^3\in\mathcal{L}_A$,
so inductively we show that a word $l$ is $k$-periodic in $\mathcal{L}_A$ if
and only if $l^2\in\mathcal{L}_A$. Thus, the third statement holds, and the
second one follows an analogous proof.
\end{proof}

We define the \textit{set of (non-constant) periods} of $\mathcal{L}$ as the
positive integers, $k$, such that there is a (non-constant) $k$-periodic word in
$\mathcal{L}$. We denote these sets by $per(\mathcal{L})$ and
$per^\ast(\mathcal{L})$ respectively.  Lemma~\ref{lem:translatingProperties}.2
(\ref{lem:translatingProperties}.3) shows that if a property $\mathcal{P}$
is expressible by forbidden (acyclic) orientations then, there is a set of words
$A$ such that for any large enough integer, $k$, the $k$-cycle belongs
to $\mathcal{P}$, if and only if $k\in per(\mathcal{L}_A)$
($k\in per^\ast(\mathcal{L}_A)$). For a set of oriented
graphs $F$ we denote by $A_F$ the set of words $\{a\in\{\la,\ra\}^\ast\colon
t(a)\in F\}$. For instance, let $F = \{TT_3, \overrightarrow{C}_3,
\overrightarrow{P}_3\}$ (note that the class of $Forb_e(F)$-graphs is
the class of bipartite graphs due to the Roy-Gallai-Hasse-Vitaver Theorem).
In this case, $A_F = \{\la\la, \ra\ra\}$, and thus the binary language
of $A_F$-free words corresponds to those words such that no two
consecutive letters are the same.\footnote{Notice that if $F$ is a finite
set,  then $A_F$ is a finite set as well, and so the binary language
corresponding to the $F$-free orientations of paths is a regular language.}

\begin{lemma}\label{lem:cycleMultiples*}
Let $F$ be a finite set of connected oriented graphs, and $m$ the maximum
order of a graph in $F$. Then, the following statements are equivalent,
\begin{itemize}
	\item there is a positive integer $k$, $k\ge\max\{4,m\}$, such that
		the $k$-cycle admits an acyclic $F$-free orientation,
	\item there is a positive integer $k$, $k\ge\max\{4,m\}$, such that
		for every multiple of $k$, $r$, the $r$-cycle admits an acyclic
		$F$-free orientation, and
	\item there is a infinite set of cycles that admit an acyclic $F$-free
		orientation.
\end{itemize}
\end{lemma}
\begin{proof}
Clearly the first item is a particular case of the third one, while the latter
is an implication of the second one. We now prove the first item implies the
second one. Let $A = A_F$, and suppose that a cycle on $k$ edges,
$k\ge \max\{4,m\}$, admits an $F$-free acyclic orientation.  By
Lemma~\ref{lem:translatingProperties}.3, there is a non-constant $k$-periodic
word $p\in\mathcal{L}_A$. Thus, for every $n\ge 1$, $p^n\in$ $\mathcal{L}_A$,
and since, for every positive integer, $l$, the equality $(p^l)^n = p^{ln}$ holds,
then $p^l$ is periodic in $\mathcal{L}_A$. Clearly $p^l$ is not a constant word,
and $|p^l|=kl$. So by Lemma~\ref{lem:translatingProperties}.3, for every multiple
of $k$, $r = kl$, the cycle on $kl$-edges admits an $F$-free acyclic orientation.
\end{proof}

The equivalent statement of Lemma~\ref{lem:cycleMultiples*} can be translated
(with the same proof) to $F$-free (not necessarily acyclic) orientations of cycles.
Moreover, there are two more equivalent statements when we do not restrict
ourselves to acyclic orientations.

\begin{lemma}\label{lem:largePaths}
Let $F$ be a finite set of oriented graphs, and $m$ the maximum order of a graph
in $F$. Then, the following statements are equivalent,
\begin{itemize}
	\item there is a positive integer $k$, $k\ge\max\{4,m\}$, such that the
		$k$-cycle admits an $F$-free orientation,
	\item there is a positive integer $k$, $k\ge\max\{4,m\}$, such that
		for every multiple of $k$, $r$, the $r$-cycle admits an $F$-free
		orientation,
	\item there is a infinite set of cycles that admit an $F$-free orientation,
	\item every path admits an $F$-free orientation, and
	\item $\mathcal{L}_A$ is infinite, where $A = A_F$.
\end{itemize}
\end{lemma}
\begin{proof}
To prove the equivalence between the first three items one can follow an
analogous proof to Lemma~\ref{lem:cycleMultiples*}. The final two statements
are equivalent due to Lemma~\ref{lem:translatingProperties}.1. To show that
the third item implies the fourth one, it suffices to notice that if a graph
$G$ admits an $F$-free orientation and $H<G$, $H$, admits an $F$-free
orientation. Since every path can be embedded in any sufficiently large
cycle, we conclude that the third item implies the fourth one. Finally, we
show that the last statement implies the first one. So we assume
there are arbitrarily large words in $\mathcal{L}_A$. Since there is only
a finite amount of $m$-words, by taking a large enough word
$w\in\mathcal{L}_A$, we can find a factor of $w$ of the form $aba$, where
$|a|\ge m$ and $b$ is possibly an empty word. Thus, by
Observation~\ref{finitesync} $ababa\in\mathcal{L}_A$, so
$abab\in \mathcal{L}_A$. By recursively using Observation~\ref{finitesync},
we prove that for any $n\ge 1$, $(ab)^n\in\mathcal{L}_A$.
Thus $(ab)^4$ is a periodic word such that $|(ab)^4|\ge 4m\ge\max\{4,m\}$,
and so by Lemma~\ref{lem:translatingProperties}.2, there is an
$F$-free orientation of a cycle on at least $\max\{4,m\}$ edges.
\end{proof}

Lemmas~\ref{lem:cycleMultiples*} and \ref{lem:largePaths} are the first results
that yield necessary conditions for a hereditary property to be expressible by
forbidden orientations, and forbidden acyclic orientations, respectively.

The following statement is a basic arithmetic result. The reader could prove it
as an exercise to not forget our basic courses of algebra and number theory,
or can refer to Appendix D \cite{collet2018} for a proof. For a set
of positive integers $A$, we denote the \textit{greatest common divisor}
of $A$ by $gcd(A)$. An integer, $l$, is a \textit{positive combination} of
a set of numbers $\{a_1,\dots,a_k\}$, if $l = m_1a_1+\dots+m_ka_k$
where $m_i$ is a positive integer for every $i\in\{1,\dots, k\}$.

\begin{lemma}\label{collet}\cite{collet2018}
For any infinite set of positive integers, $A$, with greatest common
divisor $r$, there is a finite subset $B\subseteq A$ such that
$gcd(B) = r$. Moreover, if $A$ is closed under addition, then $A$
is cofinite in $r\mathbb{Z}^+$, i.e., the complement of $A$ in $r\mathbb{Z}^+$
is finite.
\end{lemma}

These are all the results we need to proceed to Section~\ref{sec:Fgraphs}.
The three remaining results of this section build up to a language theoretic
result that is within reach now.

We say that a set of positive integers $A$ satisfies the \textit{weak addition
property}, if there are: a finite subset $B\subseteq A$, $B = \{b_1,\dots, b_k\}$,
such that $gcd(B) = gcd(A)$, and a multiple of $gcd(A)$, $l$, such that for
every positive combination of elements in $B$, $c = m_1b_1+\dots +m_kb_k$,
the integer $l+c$ belongs to $A$.

\begin{lemma}\label{cofinite}
Let $A$ be a set of positive integers with greatest common divisor
$r$. If $A$ satisfies the weak addition property, then $A$ is cofinite in
$r\mathbb{Z}^+$.
\end{lemma}
\begin{proof}
Let $B\subseteq A$ be a finite subset such that $gcd(B) = r$, let $S$ be
the set of positive combinations of elements in $B$, and let $l$ be a multiple
of $r$ such that for every $s\in S$, $l+s\in A$. Clearly, $S$ is closed under
addition and it is also not hard to notice that $gcd(S) = gcd(B) = r$. Thus,
by Lemma~\ref{collet}, $S$ is cofinite in $r\mathbb{Z}^+$, and since $l$ is
a multiple of $r$, then $l+S = \{l+s\colon s\in S\}$ is cofinite in $r\mathbb{Z}^+$.
Recall that, by the choice of $A$ and $l$, $l+S\subseteq A$, so $A$ is also
cofinite in $r\mathbb{Z}^+$.
\end{proof}

We say that a language $\mathcal{L}$ is \textit{transitive} if for every two words
$a,b\in \mathcal{L}$, there is a third (possibly empty) word $d$, such that
$adb\in\mathcal{L}$. As a temporary and convenient definition, we say that
the greatest common divisor of an empty set is $0$.

\begin{lemma}\label{lem:periods}
Let $m$ be a positive integer, let $\mathcal{L}$ be a hereditary, transitive,
$m$-synchronizing language, and let $r = gcd(per(\mathcal{L}))$. Then,
$per(\mathcal{L})$ is a cofinite subset of $r\mathbb{Z}^+$.
\end{lemma}
\begin{proof}
The case when $per(\mathcal{L})$ is empty is clear. So, we assume that
 $per(\mathcal{L})$ is not empty. We will show that $per(\mathcal{L})$ satisfies
the weak addition property, and thus conclude by Lemma~\ref{cofinite}.
By Lemma~\ref{collet}, we can choose a finite set $A$, $A=\{a_1,\dots,a_k\}
\subset per(\mathcal{L})$, such that $gcd(A) = gcd(per(\mathcal{L}))$.
We can assume that $\min\{a_1,\dots,a_k\}\ge m$. If this was not the case,
let $p_1,\dots,p_k$ be distinct primes greater than $\max\{a_1,\dots,a_k\}$,
such that $p_ia_i\ge m$ for every $i\in\{1,\dots,k\}$. Clearly,
$gdc(\{p_1a_1,\dots,p_ka_k\}) = gdc(A)$, and since for every
$1\le i\le k$,  there is a periodic word $\alpha_i$ such that $|\alpha_i|
= a_i$, then $\alpha_i^{p_i}$ is  a periodic word of length $a_ip_i$, i.e.\
$a_ip_i\in per(\mathcal{L})$. Thus, without loss of generality we will
assume that $\min\{a_1,\dots,a_k\}\ge m$. Now, let us observe that
there is a positive integer $l$ such that for any positive combination of
elements of $A$, $c=m_1a_1+\dots+m_ka_k$, there is a $(c+l)$-periodic
word in $\mathcal{L}$. For every $i\in\{1,\dots, k\}$ let $\alpha_i$ by an
$a_i$-periodic word in $\mathcal{L}$.
Since $\mathcal{L}$ is a transitive language, for any $\alpha_i$ with
$1\le i< k$ there is a word $\beta_i$ such that $\alpha_i\beta_i\alpha_{i+1}$
$\in\mathcal{L}$, and there is a word $\beta_k$ such that
$\alpha_k\beta_k\alpha_1\in \mathcal{L}$. Let $m_1a_1+\dots+m_ka_k$
be a positive combination of $A$, that is $m_i\ge 1$.  Recall
that $\mathcal{L}$ is $m$-synchronizing, and since $\min\{a_1,\dots,a_k\}$
$\ge m$, then
$\alpha_1^{m_1}\beta_1\alpha_2^{m_2}\beta_2\dots\alpha_k^{m_k}\beta_k\alpha_1$
belongs to $\mathcal{L}$. Let
$\gamma= \alpha_1^{m_1-1}\beta_1\alpha_2^{m_2}\beta_2\dots
\alpha_k^{m_k}\beta_k$,  we proceed to show that $\alpha_1\gamma$ is periodic in
$\mathcal{L}$. Since  $|\alpha_1|\ge m$ and $\alpha_1\gamma\alpha_1 =
\alpha_1^{m_1}\beta_1\alpha_2^{m_2}\beta_2\dots\alpha_k^{m_k}\beta_k
\alpha_1\in\mathcal{L}$, then  $\alpha_1\gamma\alpha_1\gamma\alpha_1
\in\mathcal{L}$. Recall that $\mathcal{L}$ is hereditary, so $\alpha_1\gamma
\alpha_1\gamma\in\mathcal{L}$,  hence, we inductively conclude that
$(\alpha_1\gamma)^n\in\mathcal{L}$. So $|\alpha_1\gamma|\in per(\mathcal{L})$.
By construction of $\gamma$, $|\alpha_1\gamma| = m_1a_1+\dots+m_ka_k + l$,
where $m_1a_1+\dots+m_ka_k$ is any positive combination of $A$, and
$l = |\beta_1|+ \dots+|\beta_k|$. Let $S$ be the set of positive combinations
of $A$. We have shown that $l+S = \{l+s\colon s\in S\}\subseteq per(\mathcal{L})$.
Thus, $per(\mathcal{L})$ satisfies the weak addition property, so by Lemma~\ref{%
cofinite}, $per(\mathcal{L})$ is a cofinite subset of $r\mathbb{Z}^+$.
\end{proof}

\begin{theorem}\label{thm:finiteperiods}
Let $A$ be a finite set of words over any alphabet.  If $\mathcal{L}_A$ is
a transitive language,  then there is a positive integer, $r$, such that
$per(\mathcal{L}_A)$ is a cofinite subset of $r\mathbb{Z}^+$.
\end{theorem}
\begin{proof}
As noted before, $\mathcal{L}_A$ is a hereditary language. Since $A$ is finite,
by Observation~\ref{finitesync}, $\mathcal{L}_A$ is $m$-synchronizing for
some $m\ge1$. Thus we conclude using Lemma~\ref{lem:periods}.
\end{proof}

\section{Expressions by forbidden (acyclic) orientations}
\label{sec:Fgraphs}

A graph (digraph) homomorphism, $\varphi:H\to G$, is an \textit{overlap} if and
only if $\varphi$ restricted to each connected component of $H$ is an
embedding. In this case we say that $G$ \textit{contains an overlap of $H$}.
Let $\mathcal{F}$ be a set of graphs, we say that a graph $G$ is
\textit{$\mathcal{F}$-overlap free} if for every $H\in\mathcal{F}$,
$G$ does not contain an overlap of $H$. Note that if $\mathcal{F}$
consists of connected graphs, then $G$ is $\mathcal{F}$-free if and
only if $G$ is $\mathcal{F}$-overlap free. Let $F$ be a set of oriented graphs,
we say that a graph $G$ admits an \textit{$F$-overlap free orientation} if there
is an orientation $G'$ such that $G'$ is $F$-overlap free as a digraph.
A property is called \textit{additive} if it is closed under disjoint unions. We denote
the disjoint union of a pair of graphs, $G$ and $H$, by $G+H$.

\begin{lemma}\label{lem:overlap}
Let $\mathcal{P}$ be a hereditary graph property with set of
minimal obstructions $\mathcal{F_P}$ and let $F$ be a set of
oriented graphs. If $\mathcal{P}$ is additive, then the
following hold:
\begin{enumerate}
	\item $\mathcal{F_P}$ consists of connected graphs,
	\item a graph $G$ is $\mathcal{F_P}$-free if and only if it is
		$\mathcal{F_P}$-overlap-free,
	\item if $Forb_e(F)$ is an expression of $\mathcal{P}$, then a graph $G$ admits
     		an $F$-free orientation if and only if $G$ admits an $F$-overlap free
		orientation, and
	\item  if $Forb_e^\ast(F)$ is an expression of $\mathcal{P}$, then a graph $G$
		admits an acyclic $F$-free orientation if and only if $G$ admits an acyclic
		$F$-overlap free orientation.
\end{enumerate}
\end{lemma}
\begin{proof}
We will prove the first statement by contrapositive. Assume that
there is a disconnected graph $H=H_1+H_2\in F_\mathcal{P}$, then
$H_1$,$H_2\in$ $\mathcal{P}$ but $H_1+H_2\not\in \mathcal{P}$, so
$\mathcal{P}$ is not additive.  Hence, $\mathcal{F_P}$ consists
of connected graphs. The second statement is a straightforward
implication of the first one.

We will prove the last two statements at once. Clearly,
if a graph $G$ admits an (acyclic) $F$-overlap free
orientation, then it admits an (acyclic) $F$-free
orientation. We will prove the remaining implication
by contrapositive, assuming the negation of $3$ ($4$)
to reach that $\mathcal{P}$ is not additive.   So,
suppose that $Forb_e(F)$ ($Forb_e^\ast(F)$) is an
expression of $\mathcal{P}$ and there is a graph $G$
that admits an (acyclic) $F$-free orientation, but not
an (acyclic) $F$-overlap free orientation.  Clearly $G
\in \mathcal{P}$. Let $F_G=\{L\in F: |V(L)|\le |V(G)|\}$,
let $k=|F_G|$, and let $l=\max\{n\in\mathbb{N}:$ there is
an oriented graph $L\in F_G$ with $n$ connected components%
$\}$. Consider the disjoint union of $G$ with itself $lk$
times, $H=\sum_{i=1}^{lk}G$, and any (acyclic) orientation
$H'$ of $H$. Naturally, every connected component of $H'$,
is an (acyclic) orientation of $G$. Since $G$ does not admit
an (acyclic) $F$-overlap free orientation, every connected
component of $H'$ is not $F_G$-overlap free. By thinking of
the elements in $F_G$ as pigeonholes, and of each connected
component of $H'$ as a pigeon, there must be an element $L
\in F_G$ that occurs as an overlap in $l$ (acyclic) oriented
copies of $G$ in $H'$. Since $L$ has at most $l$ connected
components, then $L<H'$. Thus, no (acyclic) orientation
of $H$ is $F$-free, hence $H \not \in \mathcal{P}$,  and
so $\mathcal{P}$ is not closed under disjoint unions.
\end{proof}

The following lemma strengthens the last two items of
Lemma~\ref{lem:overlap}.

\begin{lemma}\label{lem:connectedF}
Let $\mathcal{P}$ be an additive and hereditary property,
and let $F$ be a set of oriented graphs. The following
statements hold:
\begin{itemize}
	\item if $Forb_e(F)$ is an expression of $\mathcal{P}$, then there is a set of
		connected oriented graphs $F_1$, such that $Forb_e(F_1)$ is an
		expression of $\mathcal{P}$, and $|F_1|\le |F|$, and
	\item  if $Forb_e^\ast(F)$ is an expression of $\mathcal{P}$, then there is a set of
		connected oriented graphs $F_1$, such that $Forb_e^\ast(F_1)$ is an
		expression of $\mathcal{P}$, and $|F_1|\le |F|$.
\end{itemize}
\end{lemma}
\begin{proof}
If $F$ is an infinite set, we choose $F_1$ to be the set of all (acyclic)
orientations of the minimal obstructions of $\mathcal{P}$.  The fact that every
oriented graph in $F_1$ is connected follows from Lemma~\ref{lem:overlap}.1.
Clearly, $|F_1|\le |F|$, and $Forb_e(F_1)$ ($Forb_e^\ast(F_1)$) is an expression
of $\mathcal{P}$.
Now suppose that $F$ is finite, and let $f$ be the number of disconnected
oriented graphs in $F$. If $f = 0$ there is nothing to prove. We will show
that if $f > 0$, then there is a set, $F'$, of (acyclic) oriented graphs, such
that $|F'| \le |F|$, the number of disconnected oriented graphs in $F'$ is
$f-1$, and $Forb(F')$ ($Forb^\ast(F')$) is an expression of $\mathcal{P}$.
Thus, the proof will follow inductively.

By Lemma~\ref{lem:overlap}.3 (\ref{lem:overlap}.4) we can think of
$\mathcal{P}$ as the class of graphs that admit an $F$-overlap free
(acyclic) orientation. For any positive integer $n$, denote by $G_n$ the
disjoint union of graphs in $\mathcal{P}$ on at most $n$ vertices. The
following three facts are not hard to verify: first $G_n\in\mathcal{P}$, also
for any graph $G\in\mathcal{P}$ there is a positive integer, $n$, such that
$G<G_n$, and finally $G_n< G_{n+1}$. From the first fact it follows that,
for every positive integer, $n$, we can choose an $F$-overlap free (acyclic)
orientation $G'_n$ of $G_n$. Let $H$ be a disconnected (acyclic) oriented
graph in $F$. It is not hard to notice that an oriented graph
$D$ is $F$-overlap free if and only if $D$ is $(F\setminus\{H\})$-overlap free
and $D$ is $h$-free for some connected component, $h$, of $H$. Thus,
there must be a connected component, $h<H$, such that an infinite
subset of the (acyclic) orientations  $\{G'_n\}_{n\ge1}$ are $h$-free and
$(F\setminus\{H\})$-overlap free. Let $\{n_k\}_{k\ge 1}$ be the infinite sequence
of positive integers, such that the (acyclic) orientations $\{G'_{n_k}\}_{k\ge1}$
are $h$-free and $(F\setminus\{H\})$-overlap free. Since any
$(F\setminus\{H\})$-overlap free oriented graph is $(F\setminus\{H\})$-free, then
all orientations $\{G'_{n_k}\}_{k\ge1}$ are $(F\setminus\{H\}\cup\{h\})$-free.
Denote by $F'$ the set $F\setminus\{H\}\cup\{h\}$. Since for any graph,
$G\in\mathcal{P}$, there is an integer $m$ such that $G<G_{n_m}$, then we
can obtain an $F'$-free (acyclic) orientation of $G$;  anyone induced by $G'_{n_m}$.
Therefore the class of $Forb_e(F')$-graphs ($Forb_e^\ast(F')$-graphs) contains
$\mathcal{P}$. On the other hand, if a graph $G$ admits an
$F'$-free (acyclic) orientation, then this (acyclic) orientation is
$F$-free, so $G\in\mathcal{P}$. Therefore, $Forb_e(F')$
($Forb_e^\ast(F')$) is an expression of $\mathcal{P}$. Clearly
$|F'|\le|F|$ and $F'$ has $f-1$ disconnected graphs. As we previously
mentioned, both claims follow inductivley.
\end{proof}

For a hereditary property $\mathcal{P}$ we denote by $cyc(\mathcal{P})$
the set of lengths of cycles that belong to $\mathcal{P}$. We are now
ready to state some necessary conditions for an additive hereditary property
to be expressible by forbidden (acyclic) orientations.

\begin{lemma}\label{nec:multiples}
Let $\mathcal{P}$ be an additive hereditary property expressible by
forbidden (acyclic) orientations. If $cyc(\mathcal{P})$ is an
infinite set, then there is a positive integer $M$ such that for every
integer $k$, $k\ge M$, the $k$-cycle belongs to $\mathcal{P}$, if and
only if for every positive multiple, $l$, of $k$ the $l$-cycle belongs to
$\mathcal{P}$.
\end{lemma}
\begin{proof}
Let $F$ be a finite set of oriented graphs such that $Forb_e^\ast(F)$ is
an expression of $\mathcal{P}$. By Lemma~\ref{lem:connectedF} we can
assume that $F$ consists of connected oriented graphs. Let $m$ be
the maximum order of an oriented graph in $F$, and let $M=\max\{4,m\}$.
Since $cyc(\mathcal{P})$ is an infinite set, we conclude by the equivalence
of the first two items of Lemma~\ref{lem:cycleMultiples*}. The case when
$Forb_e(F)$ is an expression of $\mathcal{P}$ follows the same proof but
using the equivalence of the items of Lemma~\ref{lem:largePaths}.
\end{proof}

Furthermore, we can add one more hypothesis and show that for any
additive hereditary property $\mathcal{P}$ expressible by forbidden
orientations, there must be arbitrarily large cycles in $\mathcal{P}$.

\begin{proposition}\label{nofiniteC}
Let $\mathcal{P}$ be an additive hereditary property expressible by
forbidden orientations. If every path belongs to $\mathcal{P}$, then
$cyc(\mathcal{P})$ is an infinite set. In particular, the classes of
chordal graphs and forests are not expressible by forbidden orientations.
\end{proposition}
\begin{proof}
Assume that there is a finite set of oriented graphs $F$, such that
$Forb_e(F)$ is an expression of $\mathcal{P}$. Since $\mathcal{P}$ is
additive, by Lemma~\ref{lem:connectedF} we can assume that every
oriented graph in $F$ is connected. Since $F$ is finite, by
Lemma~\ref{lem:largePaths} there is an infinite set of cycles that
admit an $F$-free orientation.
\end{proof}

Proposition~\ref{nofiniteC} cannot be extended to properties
expressible by forbidden acyclic orientations. The intuition
behind this fact is that, depending on the set of oriented
graphs $F$, a directed path could be an $F$-free acyclic
orientation of a path, while a directed cycle is clearly
not an $F$-free acyclic orientation of a cycle (regardless
of the set of oriented graphs). For instance, chordal graphs
and forests satisfy that every path belongs to these properties,
but only a finite amount of cycles do as well, and both
classes are expressible by forbidden acyclic orientations
\cite{skrienJGT6}. This previous observation together with
Proposition~\ref{nofiniteC} show that there are some properties
expressible by forbidden acyclic orientations that are not
expressible by forbidden orientations. Now we show that there
are some natural graph classes that are not expressible by
forbidden orientations nor by forbidden acyclic orientations.
To do so, recall that given a graph, $G$, a \textit{hole} in $G$
is an induced cycle in $G$ of length greater than three.

\begin{proposition}
For an integer $k$, $k\ge 2$, let $\mathcal{P}_k$ be the class of graphs
defined as those graphs with no holes of length a multiple of $k$.
Then, $\mathcal{P}_k$ is not expressible by forbidden (acyclic) orientations. In
particular, the class of even-hole free graphs is not expressible by forbidden
(acyclic) orientations.
\end{proposition}
\begin{proof}
It suffices to notice that $\mathcal{P}_k$ is an additive hereditary
property such that $cyc(\mathcal{P})$ is an infinite set. By the
choice of $\mathcal{P}_k$, we can choose an arbitrarily large integer
$m$ such that the $m$-cycle belongs to $\mathcal{P}_k$,
but the $km$-cycle does not belong to $\mathcal{P}_k$.
So by Lemma~\ref{nec:multiples}, $\mathcal{P}_k$  is not
expressible by forbidden (acyclic) orientations.
\end{proof}

We proceed to strengthen the necessary conditions proposed in
Lemma~\ref{nec:multiples}. To do this, we introduce a technical
but not so rare property of graph classes. Consider a pair of cycles
$C$ and $C'$, we denote by $CC'$ the graph obtained by
taking the disjoint union of $C$ and $C'$, and then identifying
a vertex in $C$ with a vertex in $C'$. We call the graph $CC'$ the
\textit{coupling} of $C$ and $C'$. We say that a property is
\textit{closed under couplings}, if for every pair of cycles,
$C,C'\in \mathcal{P}$, the coupling $CC'$
belongs to $\mathcal{P}$. Finally, given a set of integers $B$ and an
integer $m$, we denote by $B_m$ the integers in $B$ greater than or
equal to $m$, and recall that given a set of oriented graphs $F$ we
denote by $A_F$ the set of words $\{a\in\{\la,\ra\}^\ast\colon t(a)
\in F\}$.

\begin{lemma}\label{sncondition}
Let $\mathcal{P}$ be an additive hereditary property expressible by
forbidden (acyclic) orientations. If $cyc(\mathcal{P})$ if an infinite set
and $\mathcal{P}$ is closed under couplings, then there is pair of
integers, $M$ and $r$, such that $cyc(\mathcal{P})_M\subseteq r\mathbb{Z}^+$,
and $cyc(\mathcal{P})_M$ is cofinite in $r\mathbb{Z}^+$.
\end{lemma}
\begin{proof}
Let $F$ be a finite set of connected oriented graphs such that $Forb_e(F)$
is an expression of $\mathcal{P}$, let $m$ be the maximum order of an
oriented graph in $F$, and let $A = A_F$. By Lemma~\ref{lem:translating%
Properties}.2, the positive integer $M$, $M=\max\{4,m+1\}$, satisfies
that if $k\geq M$, then $k\in cyc(\mathcal{P})$ if and only if $k\in
per(\mathcal{L}_A)$. Thus, $cyc(\mathcal{P})_M =  per(\mathcal{L}_A)_M$,
so for any  $r,s\in per(\mathcal{L}_A)_M$, the cycles $C_r$ and $C_s$
belong to $\mathcal{P}$. Since  $\mathcal{P}$ is closed under
couplings then there is an $F$-free orientation, $C_rC_s'$, of the
coupling $C_rC_s$. From this orientation, and from the fact that
$r,s\geq M> m$, it is not hard to obtain a periodic word in
$\mathcal{L}_A$ of length $r+s$. Indeed, by traversing $C_rC_s'$
starting by the unique vertex that belongs to both cycles, then
traversing $C_r$ and then $C_s$, we obtain an $(s+r)$-periodic word
in $\mathcal{L}_A$.
Hence, $per(\mathcal{L}_A)_M$ is closed under addition. Let
$r = gcd(per(\mathcal{L}_A)_M)$, so by Lemma~\ref{collet}, the set
$per(\mathcal{L}_A)_M$ is cofinite in $r\mathbb{Z}$, and since
$cyc(\mathcal{P})_M = per(\mathcal{L}_A)_M$, then
$cyc(\mathcal{P})_M\subseteq r\mathbb{Z}^+$, and $cyc(\mathcal{P})_M$
is cofinite in $r\mathbb{Z}^+$.

The remaining case, when $\mathcal{P}$ is expressible by forbidden acyclic
orientations, follows an analogous proof.
\end{proof}

A natural way of defining a graph class is by forbidding a set of holes.

\begin{theorem}\label{thm:main*}
Let $\mathcal{C}$ be a set of positive integers and let $\mathcal{P}$
be the set of graphs with no holes of lengths in $\mathcal{C}$. If
$\mathcal{P}$ is expressible by forbidden acyclic orientations, then
one of the following statements hold:
\begin{itemize}
	\item $\mathcal{C}$ is a finite set,
	\item $\mathcal{C}$ is a cofinite subset of $\mathbb{Z}^+$; equivalently
	$cyc(\mathcal{P})$ is a finite set, or
	\item there is a positive integer $M$ such that $\mathcal{C}_M$ is the set of
	odd integers greater than or equal to $M$; equivalently, $cyc(\mathcal{P})_M$ the
	set of even integers greater than or equal to $M$.
\end{itemize}
\end{theorem}
\begin{proof}
Assume that $\mathcal{C}$ is not a finite set, nor a cofinite subset of
$\mathbb{Z}^+$. First note that $cyc(\mathcal{P})$ is the complement
of $\mathcal{C}$ in the set of integers greater than or equal to $3$. As
$\mathcal{C}$ is not cofinite in $\mathbb{Z}^+$, then $cyc(\mathcal{P})$ is
an infinite set. Let $F$ be a set of connected oriented graphs such that
$Forb_e^\ast(F)$ is an expression of $\mathcal{P}$. By definition of $\mathcal{P}$,
$\mathcal{P}$ is closed under couplings, thus by Lemma~\ref{sncondition}
there is a pair of positive integers, $r$ and $m$, such that $cyc(\mathcal{P})_{m}$
is the set of multiples of $r$ greater or equal to $m$.  Furthermore, we can assume
that every oriented graph in $F$ has less that $m$ vertices (otherwise let $m'$
be a large enough multiple of $r$ that satisfies our assumption). We proceed
to prove that $r = 2$ by contradiction.  Since $\mathcal{C}$ is infinite,
then $r >1$, so we will assume that $r > 2$. Let $\alpha$ be an integer such
that $r\alpha > m+1$, and consider the cycle on $n$ vertices, $C_n$, where
$n = 2r\alpha - 2$. By the choice of $\alpha$, we know that $n > m$, and
since $r$ is greater than $2$, then $n$ is not a multiple
of $r$, thus $C_n$ does not belong to $\mathcal{P}$. Let $x,y\in V(C_n)$ be two
antipodal vertices and let $G = C_n + xy$. Clearly $G$ has two holes each of length
$r\alpha$. Again, by the choice of $\alpha$, $C_{r\alpha} \in \mathcal{P}$, so
$G$ contains no holes of length in $\mathcal{C}$, thus $G\in \mathcal{P}$ and
it admits an $F$-free acyclic orientation $G'$. The induced acyclic
orientation of $C_n$ by $G'$ is $F$-free since $F$ consists of
connected oriented graphs of size at most $m < r \alpha-1$. Which
contradicts the fact that $C_n \not\in \mathcal{P}$ and
$Forb_e^\ast(F)$ is an expression of $\mathcal{P}$.
\end{proof}

\begin{theorem}\label{thm:main}
Let $\mathcal{C}$ be a set of positive integers and let
$\mathcal{P}$ be the set of graphs with no holes of lengths
in $\mathcal{C}$. If $\mathcal{P}$ is expressible by forbidden
orientations, then one of the following statements hold:
\begin{itemize}
	\item $\mathcal{C}$ is a finite set, or
	\item there is a positive integer $M$ such that $\mathcal{C}_M$ is the set
	of odd integers greater than or equal to $M$. Equivalently, $cyc(\mathcal{P})_M$
	the set of even integers greater than or equal to $M$.
\end{itemize}
\end{theorem}
\begin{proof}
Once we prove that $\mathcal{C}$ is not a cofinite subset of
$\mathbb{Z}^+$, then we conclude following a proof analogous
to the one used for Theorem~\ref{thm:main*}.  By definition of
$\mathcal{P}$, every path belongs to $\mathcal{P}$. Moreover,
$\mathcal{P}$ is closed under disjoint unions. Since $\mathcal{P}$
is expressible by forbidden orientations, then, by
Proposition~\ref{nofiniteC}, $cyc(\mathcal{P})$, is an infinite
set, and so $\mathcal{C}$ cannot be a cofinite subset of
$\mathbb{Z}^+$.
\end{proof}

In particular, the class of graphs with no induced cycles of prime
length is not expressible by forbidden (acyclic) orientations. Moreover,
if $\mathcal{C}$ is any infinite set of prime numbers, then
the class of graphs with no cycles of lengths in $\mathcal{C}$
is not expressible by forbidden (acyclic) orientations, and there
are uncountable many such sets $\mathcal{C}$. Actually, we can fix
any other infinite set of positive integers (except for the set of
odd integers), and apply the previous idea together with
Theorem~\ref{thm:main} to obtain an uncountable class
of hereditary properties not expressible by forbidden
orientations.

The downside of Theorems~\ref{thm:main*}  and~\ref{thm:main} is that
they show that forbidden (acyclic) orientations have a rather weak
expressive power regarding graph classes defined by forbidding induced
cycles. But such strong necessary conditions raise our hopes of
developing any of these theorems into a  characterization.


\section{$Forb$-graphs}
\label{sec:Forb}

In the previous section we looked at hereditary properties defined
by forbidden induced cycles and exhibited necessary conditions
upon these classes to be expressible by forbidden orientations.
In this section we study homomorphism classes and propose
a characterization of those that are expressible by $Forb$-graphs.
Recall that every property expressible by $Forb$-graphs is
expressible by forbidden orientations, but not necessarily the other
way around. So regarding Question~\ref{q:basic}, the
characterization we propose in this section yields a sufficient
condition for homomorphism classes to be expressible by forbidden
orientations.

There are two main motivations to study homomorphism classes
expressible by $Forb$-graphs. On one hand, these expressions generalize
the well-known and previously mentioned Roy-Gallai-Hasse-Vitaver
Theorem. On the other hand, note that  if a property, $\mathcal{P}$, is
expressible by $Forb$-orientations, then $\mathcal{P}$ is closed under
homomorphic pre-images, and the most common properties closed under
homomorphic pre-images are homomorphism classes, i.e., classes
of $H$-colourable graphs for some fixed graph $H$.

Dually to the definition of $Forb(F)$, for a set of digraphs (graphs) $\mathcal{M}$
we denote by $CSP(\mathcal{M})$ the class of digraphs (graphs) $D$ such that
$D \to M$ for some $M\in \mathcal{M}$.  We call $CSP(\mathcal{M})$ the class
of $\mathcal{M}$-colourable digraphs (graphs). If $\mathcal{M} = \{M\}$, we will
simply write $CSP(M)$. A \textit{duality pair} in the digraph homomorphism order
is an ordered pair of digraphs $(A,B)$ such that $Forb(A) = CSP(B)$.
In \cite{nesetrilJCTB80} Ne\v{s}et\v{r}il and Tardif characterize duality pairs as
follows\footnote{Actually their result encompasses more
general relational structures, but we state it only for
the context of digraphs.}.

\begin{theorem}\label{thm:dualitypairs}\cite{nesetrilJCTB80}
If $(A,B)$ is a duality pair in the digraph homomorphism order then $A$ is
homomorphically equivalent to an oriented tree. Moreover, if $T$ is an oriented
tree, then there is a digraph $D_T$ such that $(T,D_T)$ is a duality pair.
\end{theorem}

We call $D_T$ a \textit{dual} of $T$ (any homomorphic equivalent digraph of
$D_T$ is a dual of $T$). It is not hard to observe that for every tree $T$
any of its duals $D_T$ is an oriented graph; simply note that $T$ can be
mapped to a symmetric arc, thus $D_T$ has no symmetric arcs.
With this observation we can immediately obtain a sufficient condition for a
class of $H$-colourable graphs to be expressible by $Forb$-orientations.
If $H$ is the underlying graph of a dual, $D_T$,  of an  oriented tree, $T$,
then the class of $Forb(T)$-graphs equates the class of $H$-colourable
graphs. Naturally, if $H$ is such a graph, and $H'$ is homomorphically
equivalent to $H$, then the class of $H'$-colourable graphs is expressible
by $Forb$-orientations. Turns out that the previous sufficient condition is
close to be a characterization.  We will derive this observation
from a more general result.

A \textit{generalized duality} in the digraph homomorphism order, is an ordered
pair of finite sets of incomparable digraphs $(F,\mathcal{M})$ such that
$Forb(F) = CSP(\mathcal{M})$. For a set of digraphs, $F$, we say that a digraph
$D\in F$ is \textit{minimal} (in $F$), if for every digraph $D'\in F$ such that
$D' \to D$, it holds that $D \to D'$. Generalized dualities have a similar
characterization to that of duality pairs, due to Foniok, Ne\v set\v ril
and Tardif.

\begin{theorem}\label{thm:gendualities}\cite{foniokEJC29}
If $(F,\mathcal{M})$ is a generalized duality, then every digraph in $F$ is
homomorphic equivalent to an oriented forest. Conversely, for every finite
incomparable set of oriented forests, $F$, there is a finite set of incomparable
oriented graphs $\mathcal{M}_F$ such that $(F,\mathcal{M}_F)$ is a generalized
duality.
\end{theorem}

As it happens with duality pairs, if $(F,\mathcal{M}_F)$ is a generalized duality
and $\mathcal{M}$ is the set of underlying graphs of $\mathcal{M}_F$, then
$Forb(F)$ is an expression of the class of $\mathcal{M}$-colourable graphs.
So by Theorem~\ref{thm:gendualities}, for every set of oriented forests $F$,
there is a set of graphs $\mathcal{M}$ such that a graph $G$ is a $Forb(F)$%
-graph if and only if $G$ is $\mathcal{M}$-colourable.

In general it does not hold that for any set of oriented graphs, $F$, there is
a finite set of graphs $\mathcal{M}$ such that the class of $Forb(F)$-graphs
is the same as $\mathcal{M}$-colourable graphs. For instance, let $F = \{TT_3\}$,
and note that for every finite set of graphs, $\mathcal{M}$, the chromatic number
of $\mathcal{M}$-colourable graphs is bounded. Since there are triangle free
graphs with arbitrarily large chromatic number, there is no finite set $\mathcal{M}$
such that $CSP(\mathcal{M})$ corresponds to the class of $Forb(TT_3)$-graphs.

One would like to jump to the conclusion that there is a finite set of graphs
$\mathcal{M}$ such that $Forb(F)$ is an expression of $\mathcal{M}$-colourable
graphs, if and only if $F$ is a set of oriented forests. Well, this statement
turns out to be true, but not at all obvious.  This characterization is inspired on
a result found in  \cite{kunMFCS4708}. It differs enough so that we cannot
simply cite their statement, but it is similar enough so that we can translate their
proof to this context and nomenclature. Will use the Sparse
Incomparability Lemma (for graphs) to do so. The version stated below
is due to Ne\v set\v ril and Zhu \cite{nesetrilJCTB90}.

\begin{theorem}[Sparse Incomparability Lemma]\label{thm:sparse}\cite{nesetrilJCTB90}
Let $k$ and $l$ be positive integers and let $G$ be a graph. Then, there is
a graph $G_0$ with the following properties:
\begin{itemize}
	\item there is a homomorphism $\varphi\colon G_0 \to G$,
	\item for any graph $H$ on at most $k$ vertices, $G \to H$ if and only
	if $G_0 \to H$, and
	\item $G_0$ has girth at least $l$.
\end{itemize}
\end{theorem}

The original version of the Sparse Incomparability Lemma asserts that for
every pair of non-bipartite graphs, $G$ and $H$, such that $G \to H$, and a
positive integer $l$, there is a graph $H_0$ such that $H_0 \to H$, $G$
and $H_0$ are incomparable, and $H_0$ has girth at least $l$. The version
stated in Theorem~\ref{thm:sparse} suits better our purpose.

\begin{theorem}\label{thm:CSP}
Let $F$ be a finite set of oriented graphs. There is a finite set
of graphs $\mathcal{M}$ such that $CSP(\mathcal{M})$ is the class
of $Forb(F)$-graphs if and only if $F$ is a set of oriented forests.
\end{theorem}
\begin{proof}
As observed before, one implication is trivial. We prove the interesting
implication; we assume that there  is a finite set $\mathcal{M}$ such that
$CSP(\mathcal{M})$ is the class of $Forb(F)$-graphs. Clearly, if $H$ is a
homomorphic image of an oriented graph in $F$, then $Forb(F\cup\{H\}) = Forb(F)$
and hence $Forb(F\cup\{H\})$-graphs $ = CSP(\mathcal{M})$. So we close
$F$ under homomorphic images and then keep only the cores (we denote it by
$F$ again). Among all such sets we choose $F$ to be of minimal cardinality.

If every minimal element of $F$ is an oriented forest there is nothing to prove.
So suppose there is a minimal element $F_0 \in F$ that is not a forest.  Note
that by the minimality of $F$, there is a $Forb(F\setminus\{F_0\})$-graph $G$
that is not a $Forb(F)$-graph. So there is a $Forb(F\setminus\{F_0\})$-graph,
$G$, such that $G\not \to M$ for any $M \in \mathcal{M}$, and $G$ admits
an orientation, $G'$, such that $F_0 \to G'$.  Moreover, we claim that we can
choose an orientation
$G'$ of $G$ such that any homomorphism $\varphi\colon F_0\to G'$ is injective.
To prove it, denote by $h(F_0)$ the set of all non-injective homomorphic images
of $F_0$. By the choice of $F$, for every oriented graph $D \in h(F_0)$ its core,
$D_c$, belongs to $F$. The fact that $D_c$ is the core of a non-injective
homomorphic image of $F_0$, implies that $|V(D_c)| < |V(F_0)|$. In particular,
for every $D \in h(F_0)$ there is an oriented graph, $D'\in F\setminus\{F_0\}$,
such that
$D' \to D$. So let $H$ be a graph such that for any orientation, $H'$, of $H$,
there is a non-injective homomorphism $\varphi\colon F_0\to H'$. By considering the
homomorphic image $\varphi[F_0]$ we conclude that there is an oriented graph
$D \in F\setminus\{F_0\}$ such that $D \to \varphi[F_0] \to H'$. Since $F$ is
minimal, our claim follows. So let $G$ be a $Forb(F\setminus\{F_0\})$-graph
and $G'$ an orientation of $G$ such that any homomorphism $\varphi\colon
F_0\to G'$ is injective.

By the Sparse Incomparability Lemma (Theorem~\ref{thm:sparse}) for $G$,
for $l > |V(F_0)|$, and for $k = \max\{|V(M)|\colon M\in \mathcal{M}\}$, there
is a graph $G_0$ with girth at least $l$ such that $G_0 \to G$, and for any graph
$M$ on at most $k$ vertices, $G_0 \to M$ if and only if $G \to M$. By the choice
of $k$, and since  $G \not \to M$ for any $M \in \mathcal{M}$, then $G_0$ is not
$\mathcal{M}$-colourable. By hypothesis, for any orientation, $G_0'$, of $G_0$,
there is an oriented graph  $L \in F$ such that $L \to G_0'$. Since $G_0 \to G$
choose $G_0'$ to be the orientation induced in $G_0$ by any homomorphism
$G_0 \to G$ and the previously chosen orientation of $G$, $G'$. Cleary $G_0'
\to G'$, and recall that $G$ is a $Forb(F\setminus \{F_0\})$-graph, thus there is
no homomorphism from an oriented graph in $F\setminus \{F_0\}$ to $G_0'$.
Which in turn implies that there is a homomorphism $\varphi\colon F_0 \to G_0'$,
and the fact that the girth of $G_0$ is greater than $|V(F_0)|$ implies that
every cycle in $F_0$ must be mapped to an oriented path in $G_0'$. Hence
there is non-injective homomorphism $F_0 \to G_0'\to G'$, contradicting the
choice of $G'$. Concluding that if $F_0$ is a minimal element in $F$ it must
be an oriented forest (or homomorphically equivalent to one).
\end{proof}

Theorem~\ref{thm:CSP} together with Theorem~\ref{thm:gendualities},
have two immediate consequences.

\begin{corollary}\label{cor:homclasses}
Let $H$ be any graph. The class of $H$-colourable graphs is expressible
by $Forb$-graphs, if and only if there is a set of oriented graphs
$\mathcal{M}_H$ whose underlying graphs are homomorphically equivalent
to $H$, such that $(F,\mathcal{M}_H)$ is a generalized duality for some
set of oriented forest $F$.
\end{corollary}

\begin{corollary}\label{cor:tree}
Consider a graph $H$. There is an oriented graph, $T$, such that $Forb(T)$ is
an expression of $H$-colourable graphs, if and only if $T$ is an oriented tree
and $H$ is homomorphically equivalent to the underlying graph of the dual
$D_T$.
\end{corollary}

In \cite{guzmanAR2} we showed that for any odd cycle, $C$, there
is an oriented path, $P_C$, such that a graph is $C$-colorable if
and only if it is a $Forb(P_C)$-graph. We did so by finding an
orientation, $C'$, of $C$ such that $C'$ is the dual of an oriented
path ($P_C$).   Corollary~\ref{cor:tree} shows that the problem of
characterizing the class of $H$-colorable graphs as $Forb(T)$-graphs,
is \textit{almost} equivalent to the previously mentioned technique,
i.e., it is equivalent to finding a graph, $R$, homomorphically
equivalent to $H$, and then find an orientation, $R'$, of $R$ such
that $R'$ is the dual of some oriented tree.


\section{Conclusions}
\label{sec:Conc}

As we have already mentioned in Section~\ref{sec:Fgraphs},
we believe that the necessary conditions exposed in
Theorems~\ref{thm:main*} and~\ref{thm:main} are
quite close to be sufficient as well. So we believe that
aiming to improve any of these theorems into a characterization is
a feasible problem to pursue. Also, since  a
hereditary property might be expressible by forbidden
orientations but not by $Forb$-graphs, we would like to
know if there is a homomorphism class expressible by
forbidden orientations, but not by $Forb$-graphs.
In other words, does Corollary~\ref{cor:homclasses} holds if we
replace ``expressible by $Forb$-graphs'' by ``expressible by
forbidden orientations''?

Similar to expressions by forbidden orientations, some authors
have studied \textit{expressions by forbidden ordered graphs}
\cite{damaschkeTCGT1990, feuilloleyJDM, hellESA2014}. Amongst these,
we would like to mention that the work of Feuilloley and Habib
\cite{feuilloleyJDM} stands out for being a  recent, thorough and complete
survey regarding such expressions.  Question~\ref{q:basic} can also be
posed for forbidden  ordered graphs. As far as we are concerned, there
is no example of a hereditary property that cannot be expressed by finitely
many forbidden ordered graphs, so we would like to propose this problem.

\begin{problem}\label{p:ordered}
Find an example of a hereditary property that is not expressible
by (finitely many) forbidden ordered graphs.
\end{problem}

By means of a simple combinatorial argument, one can prove that if
$\mathcal{P}$ is expressible by forbidden acyclic orientations, then
it is expressible by forbidden ordered graphs. Due to this
observation, we believe that any property that is not expressible by
forbidden acyclic orientations is a reasonable candidate to be
a witness of the example required by Problem~\ref{p:ordered}. In
particular we ask the following question.

\begin{question}
Is there a finite set of ordered graphs, $F$, such that a graph $G$
admits an $F$-free ordering, if and only if $G$ is an even-hole-free
graph?
\end{question}

Finally as a side note to the reader familiar with symbolic dynamics, form
Theorem~\ref{thm:finiteperiods} we obtain the following result.

\begin{corollary}
Let $X$ be a topological transitive shift of finite type, let $H$ be its
set of periods and let $r = gcd(H)$. Then, $H$ is a cofinite subset of
$r\mathbb{Z}^+$. In particular, if $H$ is a set of relative primes,
then $H$ is cofinite in $\mathbb{Z}^+$.
\end{corollary}
\begin{proof}
The reader familiar with shift spaces, can notice that $X$ is topological
transitive if and only if the language, $L(X)$, of $X$ is transitive. Moreover,
$X$ is of finite type if and only if there is a finite set of words $A$, such
that $L(X) =\mathcal{L}_A$. Finally, the equalities $H = per(L(X)) =
per(\mathcal{L}_A)$ hold, and $H$ is not an empty set for a shift of finite type.
Thus, Theorem~\ref{thm:finiteperiods} implies the statement of this corollary.
\end{proof}

\section{Acknowledgments}

The authors are deeply grateful to Pavol Hell for many
discussions that led to the ideas in this work.   Also,
for his feedback on a preliminary version of this work,
which helped us to greatly improve its final quality.

\end{document}